\documentclass[12pt,a4paper]{article}
\usepackage[english]{babel}
\usepackage[utf8]{inputenc}
\usepackage[T1]{fontenc}
\usepackage{graphicx}
\usepackage{url}
\usepackage{amsmath,amsthm,amssymb}
\usepackage{amsfonts}
\usepackage{amssymb}
\usepackage{esint}
\usepackage{soul}
\usepackage{floatrow}
\usepackage[a4paper,hcentering,vcentering]{geometry}
\usepackage{subeqnarray}
\usepackage[all]{xy}
\usepackage{varioref}
\usepackage{tikz}
\usepackage{hyperref}
\usepackage{pgfplots}
\DeclareMathOperator{\cat}{\text{CAT}}
\DeclareMathOperator{\sigmap}{\sigma (+ \infty)}
\DeclareMathOperator{\sigmam}{\sigma(- \infty)}
\DeclareMathOperator{\supp}{\text{supp}}
\DeclareMathOperator{\haar}{Haar}
\DeclareMathOperator{\ugm}{\it{U^{-}_g}}
\DeclareMathOperator{\ugp}{\it{U^{+}_g}}
\DeclareMathOperator{\uhm}{\it{U^{-}_h}}
\DeclareMathOperator{\uhp}{\it{U^{+}_h}}
\DeclareMathOperator{\iso}{Isom}

\newcommand{\bd}{\partial_\infty}
\newcommand{\nui}{\check{\nu}}
\newcommand{\mui}{\check{\mu}}
\newcommand{\pmui}{P_{\mui}}
\newcommand{\prob}{\text{Prob}}
\theoremstyle{plain}
\newtheorem{thm}{Theorem}[section]
\newtheorem*{Rank thm}{Rank Rigidity Theorem}
\newtheorem*{Rank cjt}{Rank Rigidity Conjecture}
\newtheorem{cor}[thm]{Corollary}
\newtheorem{lem}[thm]{Lemma}
\newtheorem{prop}[thm]{Proposition}
\theoremstyle{definition} 
\newtheorem{Def}[thm]{Definition}
\newtheorem{rem}[thm]{Remark}
\theoremstyle{definition}
\newtheorem*{ackn}{Acknowledgement}

\title{Random walks and rank one isometries on CAT(0) spaces}
\author{Corentin Le Bars}
\date{\vspace{-5ex}}

\begin{document}
	\maketitle
	
	\section{Introduction}	
	
	Let $G$ be a discrete countable group and $\mu \in \prob (G)$ a probability measure on $G$. We define the random walk $(Z_n)_n$ on $G$ as the sequence $Z_n := \omega_1 \dots \omega_n$, where the $\omega_i'$s are chosen independently according to the law $\mu$. Fixing a point $x \in X$, where $X$ is a given metric space on which $G$ acts by isometries, we want to study the sequence of points $(Z_n x)_n$. In this paper, we are particularly interested in the asymptotic behaviour of this random walk, and if the space $X$ has the right geometric features, one can hope that $(Z_n x)_n$ is going to converge almost surely in a natural compactification of $X$. The typical setting in which results have been obtained is when $X$ is of negative curvature, or at least when $X$ has some kind of hyperbolic properties. In the fundamental paper of V. Kaimanovich \cite{kaimanovitch00}, the convergence of $(Z_n x)_n$ to a point of the visual boundary is proven for groups acting geometrically on proper hyperbolic spaces and several other classes of actions. More recently this result has been extended by J. Maher and G. Tiozzo in \cite{maher_tiozzo18} for groups acting by isometries on non proper hyperbolic spaces. We emphasize on the fact studying random walks on non proper spaces involves different techniques, as in the non proper case the space $\overline{X}$ is no longer a compactification of $X$: $\overline{X}$ might be non compact, and $X$ is no longer open in $\overline{X}$. In the sequel, we will only consider proper metric spaces. 
	\newline
	
	In our case, $X$ is a proper $\cat(0)$ space, and $G$ acts on $X$ by isometries, but not necessarily cocompactly nor properly. If the reader wants a detailed introduction to $\cat(0)$ geometry, standard references are \cite{bridson_haefliger99} and \cite{ballman95}. The leading examples to have in mind when thinking about $\cat(0) $ spaces are Hadamard manifolds, $ \cat(0)  $ cube complexes and buildings. We simply recall that there exists a natural compactification $\overline{X}$ of $X$, for which every point of the boundary is represented by some class of asymptotic geodesic rays. In particular, for $(g_n)_n $ a sequence of isometries in $G$, the convergence of $(g_n x)_n$ to a point $\xi \in \bd X$ does not depend on the basepoint $x$. In this paper, we show that the random walk $(Z_n x)_n$ converges almost surely to a point of the boundary, provided that there exist rank one elements in the group. A rank one element is an axial isometry whose axes do not bound any flat half plane. Due to \cite[Theorem 5.4]{bestvina_fujiwara09}, an isometry is rank one if one of its axis $\sigma$ is "contracting", namely there exists a constant $K$ such that the projection onto $\sigma$ of every geodesic ball disjoint from $\sigma$ has diameter less than $K$. This type of behaviour typically appears in hyperbolic spaces, where any loxodromic isometry is contracting in this sense. Then, it is often fruitful to think about rank one isometries as isometries that satisfy hyperbolic-like properties. Another important feature of rank one isometries is that they act on the boundary with North-South dynamics \cite[Lemma III. 3. 3]{ballman95}. Namely, for $g$ a rank one isometry of $X$, there exist two points $g^{+} $ and $g^{- }$ in $ \bd X$ such that the successive powers of $g$ contract the whole boundary $\bd X$ minus $g^{-}$ on $g^{+}$. In Section \ref{rank one section}, we review some of these results and their consequences. 
	\newline
	
	The study of rank one elements is motivated by the Rank Rigidity theory for Riemannian manifolds of non positive curvature, due to W. Ballmann, M. Brin, R. Spatzier, P. Eberlein and K. Burns (see \cite{ballmann_brin95}, \cite{ballman_brin_eberlein}, \cite{burns_spatzier}, \cite{eberlein_heber} and \cite{ballman_burns_spatzier}). For a detailed introduction and proof of the Rank Rigidity Theorem for Hadamard manifolds, see \cite{ballman95}. The Theorem states that if $M$ is a Hadamard manifold, and if $G$ is a discrete group acting properly and cocompactly on $M $, then either $M$ decomposes as a non-trivial product of two manifolds, or $M$ is a higher rank symmetric space and $G$ contains a rank one isometry. 
	
	This alternative appears to hold for wider classes of $\cat(0) $ spaces, which led to formulate a general conjecture. We state one of the possible formulations as it is written in \cite{caprace_sageev11}. Recall that a metric space is said geodesically complete if any geodesic can be extended to infinity.  
	
	\begin{Rank cjt}
		Let $X$ a locally compact geodesically complete $\cat(0)$ space, and $G$ a discrete group acting properly and cocompactly on $X$ by isometries. Assume that $X$ is irreducible. Then $X$ is either a higher rank symmetric space or a Euclidean building of dimension $\geq 2$, or $G$ contains a rank one isometry. 
	\end{Rank cjt}
	
	Over the last thirty years, the conjecture was proven to hold for Euclidean cell complexes of dimension 2 and 3 (Ballmann and Brin \cite{ballmann_brin95}), and P-E. Caprace and K. Fujiwara have proven that it also holds within the class of buildings and Coxeter groups \cite{caprace_fujiwara10}. More recently, P-E. Caprace and M. Sageev have proven that it remains true for finite-dimensional $\cat (0) $ cube complexes \cite{caprace_sageev11}. 
	\newline
	
	It is to be noted that rank one elements may play important roles in other fields of research, among which the Tits Alternative conjecture, stating that every finitely generated subgroup of a group acting geometrically on a $\cat(0)$ space either contains a nonabelian free subgroup or is virtually solvable. The Tits alternative is known in many cases, see \cite{caprace_sageev11}, \cite{osajda_przytycki21} and references therein. More recently, D.~Osajda and P.~Przytycki have proven that it holds in the context of an almost free action on a $\cat(0)$ triangle complex, which we do not assume locally compact \cite[Theorem A]{osajda_przytycki21}. In the proof, they give a geometric criterion on the $\cat(0)$ triangle complex for finding rank one elements in the group \cite[Proposition 7.3]{osajda_przytycki21}, and hence free subgroups. Note that in this situation, the action needs not be cocompact. 
	\newline
	
	Last, note that the presence of rank one elements in a group acting on a $\cat(0)$ space can be detected by general conditions on the group action along with some purely geometric features of the space, for example concerning the Tits boundary of the space, see Section \ref{tits metric subsection}.
	\newline
	
	A measure $\nu $ on $X$ is said to be $\mu$-stationary if $\mu \ast \nu = \nu $, where $\mu \ast \nu$ defines the convolution measure of $\mu $ and $\nu$. When we want to study random walks generated by a measure $\mu$ on $G$, it is often fruitful to endow $\overline{X}$ with a probability measure which is stationary with respect to $\mu$. It allows to use important results in measure theory including those of H. Furstenberg \cite{furstenberg73}. They will be reviewed in Section \ref{stationary section}. 
	\newline 
	
	The first result of this paper is the fact that in our context, there is a unique $\mu$-stationary measure on $\overline{X}$.

	\begin{thm}\label{measure thm}
		Let $G$ be a discrete group and $G \curvearrowright X$ a non-elementary action by isometries on a proper $\cat (0)$ space $X$. Let $\mu \in \prob(G) $ be an admissible probability measure on $G$, and assume that $G $ contains a rank one element. Then there exists a unique $\mu$-stationary measure $\nu \in \prob (\overline{X})$. 
	\end{thm}
	
	Theorem \ref{measure thm} is fundamental in order to obtain the almost sure convergence of the random walk $(Z_n x)_n$ to the boundary. However, we think that the uniqueness of the stationary measure can be of independent interest. The second result, and the main Theorem of this article is the almost sure convergence to the boundary.
	
	\begin{thm}\label{convergence thm}
		Let $G$ be a discrete group and $G \curvearrowright X$ a non-elementary action by isometries on a proper $\cat (0)$ space $X$. Let $\mu \in \prob(G) $ be an admissible probability measure on $G$, and assume that $G $ contains a rank one element. Then for every $x \in X$, and for $\mathbb{P}$-almost every $\omega \in \Omega$, the random walk $(Z_n (\omega) x)_n $ converges almost surely to a boundary point $z^{+}(\omega) \in\bd X$. Moreover, $z^{+}(\omega)$ is distributed according to the stationary measure $\nu$. 
	\end{thm} 
	
	Results of convergence of this type had already been obtained for special cases. In the context of a fundamental group $\pi_1(M)$ of a compact rank one Riemannian manifold $M$ acting on the sphere at infinity $\partial \widetilde{M}$ of the universal covering space $\widetilde{M}$, the convergence of the sample paths was proven by Ballmann \cite[Theorem 2.2]{ballmann89}. In the context of finite dimensional $\cat(0)$ cube complexes, the behaviour of the sample paths of a given group has been extensively studied by T.~Fern\'os, J.~Lécureux and F.~Mathéus who showed under weak hypotheses that $(Z_n x)_n$ converges almost surely  to a point of the visual boundary \cite[Theorem 1.3]{fernos_lecureux_matheus18}, and who gave an extensive description of the asymptotic behaviour of the random walk: nature of the limit points, occurrence of the contracting elements... 
	
	It is also to be noted that the techniques we used in order to prove Theorem \ref{measure thm} can be rewritten in view of Theorem \ref{convergence thm} as the following: 
	
	\begin{cor}
		Let $\xi \in \bd X$ be a limit of the random walk $(Z_n x)_n$. Then for $\nu$-almost every point $\eta \in \bd X$, there exists a rank one geodesic joining $\xi $ to $\eta$. 
	\end{cor}
	
	In other words, limit points are almost surely rank one. This result will be useful for the proof of Theorem \ref{drift thm}, but we think it could be used in different contexts. 
	
	In the more general setting of a non-specified $\cat(0)$ space, Karlsson and Margulis had already proven in \cite[Thoerem 2.1]{karlsson_margulis} a first general result of convergence of the random walk. In fact, they studied the more general behaviour of cocycles on $X$, but it is a problem that we will not consider here. When the measure $\mu$ has finite first moment $\int_G d(g x, x) d\mu(g) < \infty$, the subadditive ergodic Theorem implies that the limit $\lambda:= \lim_{n} \frac{1}{n}d(Z_n x, x)$ exists almost surely. This limit is called the \textit{drift} of the random walk, and can be understood as the speed at which the random walk goes to infinity. Since the action is isometric, $\lambda$ does not depend on the choice of the basepoint. Under the hypothesis that the drift is positive, Karlsson and Margulis had showed, among other results, that the random walk $(Z_n x)_n$ converges almost surely to a point of the visual boundary. Theorem \ref{convergence thm} is different because we don't assume that the measure $\mu $ has finite first moment, nor that the drift is positive. In the case that $G$ is non-amenable, Guivarc'h showed that the random walk generated by a word metric has positive drift \cite{guivarch80}, but it can be quite difficult to prove in general. 
	\newline

	When we assume that the measure $\mu$ has finite first moment, the drift $\lambda$ exists, and another important subject concerning the asymptotic behaviour of the random walk is knowing whether $\lambda$ is positive or not. When the group $G$ is non-amenable and endowed with some word metric, the drift is positive, but for general actions, the answer is not clear. In \cite{karlsson_margulis}, the positivity of the drift was used to prove the convergence to the boundary, while here we obtain this result after the convergence. 
	
	\begin{thm}\label{drift thm}
		Let $G$ be a discrete group and $G \curvearrowright X$ a non-elementary action by isometries on a proper $\cat (0)$ space $X$. Let $\mu \in \prob(G) $ be an admissible probability measure on $G$ with finite first moment, and assume that $G $ contains a rank one element. Let $x \in X$ be a basepoint of the random walk. Then the drift $\lambda$ is almost surely positive: 
		\begin{equation}
			\lim_{n \rightarrow \infty} \frac{1}{n} d(Z_n(\omega) x, x) = \lambda >0.
		\end{equation}
	\end{thm}
	\vspace{5mm}
	In order to prove Theorem \ref{measure thm}, we use a result from Papasoglu and Swenson about the dynamics on the Tits boundary \cite[Lemma 19]{papasoglu_swenson09}, stating that in the $\cat(0)$ case, the action of the group $G$ satisfies a $\pi$-convergence property: if $(g_n)_n$ is a sequence of isometries satisfying $g_n x \underset{n \rightarrow \infty}{\longrightarrow} \xi \in \bd X$ and $g_n^{-1} x \underset{n \rightarrow \infty}{\longrightarrow} \eta \in \bd X$, then for all compact subset $K \subseteq \partial X  - B_T(\eta, \pi)$, and all open set $U $ containing $\xi$, $g_n K \subseteq U $ for all $k $ large enough. It will then be useful to prove that the $\nu$-measure of $B_T(\eta, \pi)$ is zero (Lemma \ref{zero measure}), which we do by using a result from Maher and Tiozzo \cite[Lemma 4.5]{maher_tiozzo18}. Once we know that the stationary measure is unique, proving that the random walk $(Z_n x )_n$ converges to the boundary (Theorem \ref{convergence thm}) follows from a geometric result concerning rank one geodesics proven by Ballmann \cite[Lemma III.3.1]{ballman95}. Last, to show that the drift is positive (Theorem \ref{drift thm}), we have followed the strategy implemented in Guivarc'h and Raugi \cite{guivarch_raugi85}, see also \cite{benoist_quint16} and \cite{benoist_quint}. 
	\newline
	
	While we were writing this paper, it came to our attention that H. Petyt, D.~Spriano and A.~Zallum have proposed another approach to the subject of $\cat(0)$ actions. More precisely, given an action of a group $G$ on a $\cat(0)$ space $X$, it is possible to build a hyperbolic space $(X_L, d_L)$ out of $X$ using "curtains", in such a way that informations on the original action can be nicely translated to the actions on the derived hyperbolic space. Considering the results we already have for actions on hyperbolic spaces (e.g. \cite{maher_tiozzo18}), it is likely that we could use these results to deduce some of the properties we have investigated in this paper about random walks one $\cat(0)$ spaces.
	\newline
	
	Section \ref{background} is an introduction to the notions that we are going to use, and presents the general setting. In Section \ref{uniqueness measure section}, we prove Lemma \ref{zero measure} and Theorem \ref{measure thm}. In Section \ref{convergence random walk}, we prove Theorem \ref{convergence thm} stating that the random walk is convergent, which is the main result of this article. In Section \ref{drift section}, we give applications of the convergence, especially the positivity of the drift and geodesic tracking results.

	\begin{ackn}
		The author is grateful to Jean Lécureux for the weekly conversations and commentaries, and whose contribution to this article was invaluable. 
	\end{ackn}

	\section{Background}\label{background}
	
	\subsection{Random walks and general setting}
	
	Let $G$ be a discrete countable group and $\mu \in \prob(G)$ a probability measure on $G$. Throughout the article we will assume that $\mu $ is admissible, i.e. $\supp(\mu)$ generates $G$ as a semigroup. Let $(\Omega, \mathbb{P}) $ be the probability space $(G^{\mathbb{N}}, \delta_e \times \mu^{\mathbb{N^\ast}})$. The application 
	\begin{equation*}
		(n, \omega) \in \mathbb{N} \times \Omega \mapsto Z_n(\omega) = \omega_1 \omega_2 \dots \omega_n,
	\end{equation*}
	where $\omega$ is chosen according to the law $\mathbb{P}$, defines the random walk on $G$ generated by the measure $\mu$.

	Let now $(X,d)$ be a proper $\cat(0)$ metric space, on which $G$ acts by isometries. If the reader wants a detailed introduction to $\cat (0) $ spaces, the main references that we will use are \cite{bridson_haefliger99} and \cite{ballman95}. We recall that the boundary $ \bd X$ of a $\cat (0) $ space $X$ is the set of equivalent classes of rays $\sigma : [0, \infty) \rightarrow X$, where two rays $\sigma_1, \sigma_2$ are equivalent if they are asymptotic, i.e. if $d(\sigma_1(t), \sigma_2 (t))$ is bounded uniformly in $t$. 
	
	Given two points on the boundary $\xi$ and $\eta$, if there exists a geodesic line $\sigma : \mathbb{R} \rightarrow X$ such that the geodesic ray $\sigma_{[0, \infty)}$ is in the class of $\xi$ and the geodesic ray $t \in [0, \infty) \mapsto \sigma(-t)$ is in the class of $\eta$, we will say that the points $\xi $ and $\eta$ are joined by a geodesic line. The reader should be aware that in general, such a geodesic need not exist between any two points of the boundary, as can be seen in $\mathbb{R}^2$. A point $\xi$ of the boundary that is called a \textit{visibility point} if, for all $\eta \in \bd X - \{\xi\}$, there exists a geodesic from $\xi$ to $\eta$. We will see in the next section a criterion to prove that a given boundary point is a visibility point. 
	
	An important feature in $\cat(0)$ spaces is the existence of closest-point projections on convex subsets. More precisely, given a closed convex subset $C$ in a proper $\cat(0)$ space, there exists a map $p_C : X \rightarrow C$ such that $p(x)$ minimizes the distance $d(x,C)$  \cite[Proposition 2.4]{bridson_haefliger99}. This map is a retraction of $X$ onto $C$ and is distance decreasing: for all $x, y \in X$, 
	\begin{equation*}
		d(p_C(x), p_C(y )) \leq d(x,y). 
	\end{equation*}
	Now, given a closed ball $B:= \overline{B}(x_0, r)$, the projection $p_r : X \rightarrow \overline{B}(x_0, r) $ can actually be extended to $\overline{X}$, by identifying any point $\xi$ of the boundary with the geodesic ray $\sigma$ issuing from $x_0$ in the class of $\xi$. In this setting, if $\sigma(0) = x_0$, we define $p_B(\xi) = \sigma (r)$. 
	Following the notations in \cite[Chapter II.8]{bridson_haefliger99}, the visual topology on $\overline{X}$ is defined by a basis of open sets $U(c, r, \varepsilon)$, where $c$ is a geodesic ray, $r,\varepsilon>0$, and  
	\begin{equation*}
		U(c, r, \varepsilon) := \{ x \in \overline{X} \ | \ d(x, c(r)) >r, d(p_r(x), c(r)) < \varepsilon)\}, 
	\end{equation*}
	where we called $p_r$ the projection on the closed (convex) ball $\overline{B}(c(0), r)$ of centre $c(0)$ and radius $r$. Given $x \in X$ and $\xi \in \overline{X}$, there is a unique geodesic ray (or segment) $c$ joining $x$ to $\xi$, so we will write alternatively $U(x, \xi, r, \varepsilon)$ for $U(c, r, \varepsilon)$. 
	
	The following proposition is taken from \cite[Proposition II. 8.8]{bridson_haefliger99}, and implies that the visual topology does not depend on the basepoint. It will be useful later. 
	
	\begin{prop}[{\cite[Proposition II. 8.8]{bridson_haefliger99}}]\label{continuite proj}
		Let $x, x' \in X$, and $r >0 $. let $c$ and $c'$ be the geodesic rays issuing from $x$ and $x'$ respectively such that $c(\infty)=c'(\infty) = \xi$. Let $p_r : \overline{X} \rightarrow \overline{B}(x, r)$ be the projection of $\overline{X}$ onto $\overline{B}(x, r)$. Then, for all $\varepsilon >0 $, there exists $R= R(r, d(x, x'), \varepsilon)>0 $ such that for all $R' \geq R$, $p_{r} (U(c', R', \varepsilon/3)) \subseteq B(c(r), \varepsilon)$. In particular, for all $R' \geq R$, $U(c', R', \varepsilon/3)$ is contained in $U(c, r, \epsilon )$. 
	\end{prop}
	\begin{figure}
		\centering
		\begin{center}
			\begin{tikzpicture}[scale=1.5]
				\draw (0,0) -- (3, 3)  ;
				\draw (1, 2) to[bend right = 50] (3, 1) node[above right] {$U(x, \xi, r, \epsilon )$};
				\draw (2, 2.7) to[bend right = 80] (3, 2) node[above right]  {$U(x', \xi, R', \varepsilon/3)$};
				\draw (1, -1)  to[bend left=10] (3,3)  ;
				\draw (0,0) node[below left]{$x$} ;
				\draw (1,-1) node[below left]{$x'$} ;	
				\draw (3,3) node[above]{$\xi$};
			\end{tikzpicture}
		\end{center}
		\caption{Illustration of Proposition \ref{continuite proj}.}
	\end{figure}
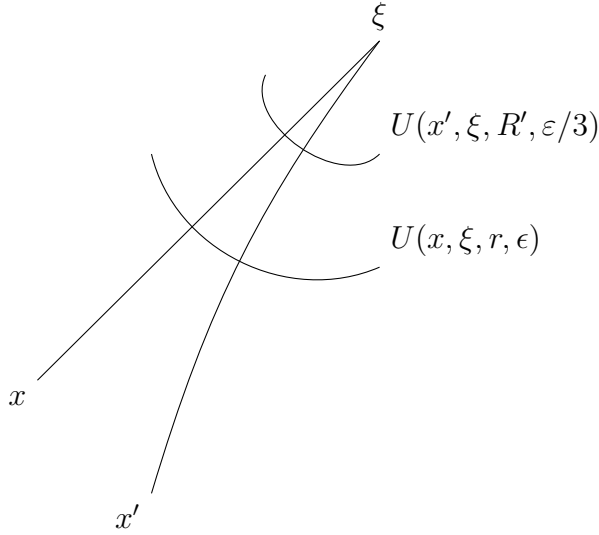

	When $X$ is a proper space, the space $\overline{X} = X \cup \partial X $ is a compactification of $X$, that is, $\overline{X}$ is compact and $X$ is an open and dense subset of $\overline{X}$. We recall that the action of $G$ on $X$ extends to an action on $\bd X$ by homeomorphisms. 
	
	Another equivalent construction of the boundary can be done using horofunctions. If $x_n \rightarrow \xi \in \bd X$, we denote by $h_\xi^{x} : X \mapsto \mathbb{R}$ the horofunction given by 
	\begin{equation*}
		h_\xi^{x}(z) = \lim_n d(x_n, z) - d(x_n, x).
	\end{equation*}
	It is a standard result in $\cat(0)$ geometry (see for example \cite[Proposition II.2.5]{ballman95}) that this limit exists and that given any basepoint $x$, a horofunction characterizes the boundary point $\xi $.

	\subsection{Rank one elements}\label{rank one section}
	
	Let $g \in G$. We say that $g$ is a \textit{semisimple} isometry if its displacement function $ x \in X \mapsto \tau_g(x) = d(x , gx)$ has a minimum in $X$. If this minimum is non-zero, it is a standard result (see for example \cite[Proposition II.3.3]{ballman95}) that the set on which this minimum is obtained is of the form $C \times \mathbb{R}$, where $C$ is a closed convex subset of X. On the set $\{c\}\times \mathbb{R}$ for $c \in C$, $g$ acts as a translation, which is why $ g$ is called \textit{axial} and the subset $\{c\}\times \mathbb{R}$ is called an \textit{axis} of $g$. A \textit{flat half-plane} in $X$ is defined as a euclidean half plane isometrically embedded in $X$.
	
	\begin{Def}
		We say that a geodesic in $X$ is \textit{rank one} if it does not bound a flat half-plane. If $g$ is an axial isometry of $X$, we say that $g$ is rank one if no axis of $g$ bounds a flat half-plane. 
	\end{Def}
	
	\begin{rem}\label{flat strip}
		Let $g$ be a rank one isometry, and let $\sigma$ be one of its axes. Then there exists $R \geq 0 $ such that $\sigma $ does not bound a flat strip of width $R$. 
	\end{rem}
	
	More information on rank one isometries and geodesics can be found in \cite[Section III. 3]{ballman95}, and more recently in \cite{caprace_fujiwara10} and in \cite{bestvina_fujiwara09}. If $X$ is a proper space, M.~Bestvina and K.~Fujiwara showed in \cite{bestvina_fujiwara09} that an isometry is rank one if and only if it induces a contraction property on its axes. More precisely, an isometry of $X$ has rank one if and only if there exists $C \geq 0$ such that one of its axis $\sigma$ is $C$-contracting: for every metric ball $B$ disjoint from the geodesic $\sigma$, the projection $\pi_\sigma (B)$ of the ball onto $\sigma$ has diameter at most $C$. 
	
	\begin{Def}
		We say that the action $G \curvearrowright X$ of a rank one group $G$ on a $\cat (0) $ space $X$ is \textit{non-elementary} if $G$ neither fixes a point in $\bd X$ nor stabilizes a geodesic line in $X$. 
	\end{Def}
	
	To justify this definition, we use a result from Caprace and Fujiwara in \cite{caprace_fujiwara10}. What follows comes from the aforementioned paper. 
	
	\begin{Def}
		Let $g_1, \, g_2 \in G$ be axial isometries of $G$, and fix $x_0 \in X$. The elements $g_1, g_2 \in G$ are called independent if the map 
		\begin{equation}
			\mathbb{Z} \times \mathbb{Z} \rightarrow [0, \infty) : (m,n) \mapsto d(g_1^m x_0, g_2^nx_0)
		\end{equation}
		is proper. 
	\end{Def}
	
	\begin{rem}
		In particular, the fixed points of two independent axial elements form four distinct points of the visual boundary. 
	\end{rem}
	
	The following result was proven by P-E.~Caprace and K.~Fujiwara in \cite{caprace_fujiwara10}. 
	
	\begin{prop}[{\cite[Proposition 3.4]{caprace_fujiwara10}}]\label{non elem caprace fuj}
		Let $X$ be a proper $\cat (0) $ space and let $G < \iso (X)$. Assume that $G$ contains a rank one element. Then exactly one of the following assertions holds: 
		\begin{enumerate}
			\item \label{elem} $G$ either fixes a point in $\bd X$ or stabilizes a geodesic line. In both cases, it possesses a subgroup of index at most 2 of infinite Abelianization. Furthermore, if $X$ has a cocompact isometry group, then $\overline{G} < \iso (X)$ is amenable. 
			
			\item \label{non elem} G contains two independent rank one elements. In particular, $\overline{G}$ contains a discrete non-Abelian free subgroup. 
		\end{enumerate}
	\end{prop}
	
	As a consequence, the action $G \curvearrowright X$ of a rank one group $G$ on a $\cat (0) $ space $X$ is non-elementary if and only if alternative \ref{non elem} of the previous Proposition holds. 
	
	The next Lemma comes from \cite{hamenstadt09}, and extends a result from Ballmann and Brin \cite{ballmann_brin95}. It is a fundamental result on the dynamics induced by rank one isometries. 
	
	\begin{thm}[{\cite[Lemma 4.4]{hamenstadt09}}] \label{hamenstadt}
		An axial isometry $f$ in $G$ is rank one if and only if $f$ acts with North-South dynamics with respect to its fixed points $f^{-}$ et $f^{+}$ : for every neighbourhood $V$ of $f^{-}$ and $U$ of $f^{+}$, there exists $k_0 \geq 0 $ such that $f^{k} (\bd X  - V) \subseteq U$  and $f^{-k} (\bd X  - U) \subseteq V$ for all $k \geq k_0$. 
	\end{thm}
	
	\subsection{Tits metric on the boundary}\label{tits metric subsection}
	
	Let us now recall some results about Tits geometry. It is a useful tool to detect flats in $\cat(0)$ spaces. The following definitions and properties can be found in \cite[Section II.4]{ballman95}, and in \cite{bridson_haefliger99}. 
	
	Let $\sigma_1 : [0, C_1] \rightarrow X$ and $ \sigma_2 : [0, C_2] \rightarrow X$ be two geodesic segments in $X$ emanating from the same basepoint $\sigma_1 (0 ) = \sigma_2(0) = x$. For every $(s, t) \in (0, C_1) \times (0,C_2)$, there exists a euclidean comparison triangle $\overline{\Delta}_{s,t}$ of the triangle $\Delta_{s,t} $ in $X$ spanned by $(x, \sigma_1 (s), \sigma_2(t))$. Write $\overline{\angle}_{\overline{x}} (\overline{\sigma_1} (s), \overline{\sigma_2} (t) )$ the angle at $\overline{x}$ of the comparison triangle $\overline{\Delta}_{s,t}$. Then by the $\cat (0)$ inequality,  $\overline{\angle}_{\overline{x}} (\overline{\sigma_1} (s), \overline{\sigma_2} (t) )$ is monotonically decreasing and we can define 
	\begin{equation*}
		\angle (\sigma_1, \sigma_2) = \lim_{s, t \rightarrow 0} \overline{\angle}_{\overline{x}} (\overline{\sigma_1} (s), \overline{\sigma_2} (t) ). 
	\end{equation*}
	
	Since $X$ is uniquely geodesic, for any triple $x, y, z \in X$, there exist exactly one geodesic segment $\sigma_1$ (resp.$ \sigma_2$) from $x$ to $y$ (resp. from $x$ to $z$), and we define the angle $\angle_x (y,z) $ at $x$ between $y$ and $z$  as $ \angle (\sigma_1, \sigma_2)$. 
	If $x_n \rightarrow \xi \in \bd X$, $y_p \rightarrow \eta \in \bd X $ in the visual topology, one can extend the notion of angle between points in the boundary by 
	\begin{equation*}
		\angle_x (\xi, \eta) = \lim_{n, p \rightarrow \infty} \angle_x (x_n, y_p). 
	\end{equation*}
	It turns out that the $\angle_x (\xi, \eta)$ does not depend on the choice of the sequences $(x_n)_n$ and $(y_p)_p$. 
	Finally, we define $\angle : \overline{X} \times \overline{X} \rightarrow [0, \pi] $ the angular metric on $\overline{X} $ by 
	\begin{equation*}
		\angle (\xi, \eta) = \sup_{x \in X} \angle_x(\xi, \eta). 
	\end{equation*}
	\begin{rem}
		For example, given two points on the boundary $\xi$ and $\eta$, if there exists a geodesic $\sigma $ joining them, then the above supremum is attained on $\sigma$ and $\angle(\xi, \eta) = \pi$. 
	\end{rem}
	
	The Tits metric on the boundary $d_T : \overline{X} \times \overline{X} \rightarrow \mathbb{R} \cup \{\infty\}$ is the length metric associated to $\angle (.,.)$. We denote by $B_T(\xi, r)$ the closed  ball of radius $r$ and centre $\xi$, defined by $B_T(\xi, r) := \{ \eta \in \bd X : d_T(\xi, \eta) \leq r\}$. 
	The following theorem summarizes important properties of the Tits metric.  
	
	\begin{thm}[{\cite[Theorem II.4.11]{ballman95}}]\label{tits}
		Let $X$ be a proper $\cat (0)$ space. Then $(\partial_T X, d_T)$ is a complete $\cat (1) $ space. Moreover, for all $\eta, \, \xi \in \bd X$ : 
		\begin{enumerate}
			\item if there is no geodesic in $X$ from $\xi $ to $\eta$, then $d_T(\xi, \eta) = \angle (\xi, \eta) \leq \pi$. 
			\item If $\angle (\xi, \eta) < \pi$, there is no geodesic in $X$ joining $\xi $ to $\eta$ and there exists a unique geodesic (for the Tits metric) from $\xi$ to $\eta $ in $\partial_T X$. 
			\item If there is a geodesic $\sigma$ in $X$  from $\xi $ to $ \eta$, then $d_T(\xi, \eta) \geq \pi $, with equality if and only if $\sigma$ bounds a flat half-plane. 
			\item \label{semi continuite} If $(\xi_n)$ and $(\eta_n)$ are two sequences in $\bd X$, such that $\xi_n \rightarrow \xi \in \bd X$ and $\eta_n \rightarrow \eta\in \bd X$ in the visual topology, then $d_T(\xi, \eta) \leq \liminf_{n\rightarrow \infty} d_T(\xi_n, \eta_n)$. 
		\end{enumerate}
	\end{thm}
	
	\begin{rem}
		In fact, for $0 \leq r< \infty$, $B_T(\xi, r)$ is closed for the visual topology. Indeed, let $(\xi_n) \subseteq B_T(\xi, r)$ such that $\xi_n \rightarrow z \in \bd X$ in the visual topology. By lower semicontinuity, $\liminf d_T(\xi, \xi_n) \geq d_T(z, \xi)$, hence $d_T(z, \xi) \leq r$. In particular, the ball $B_T(\xi, r)$ is $\nu$-measurable.
	\end{rem}
	
	Let now $g \in G$ be a rank one element and let $\sigma$ be an axis of $g$. Then $\sigma(+\infty) $ and $\sigma(- \infty)$ are visibility points of the boundary. In particular, $d_T (\sigmam, \xi) = + \infty$ for all $\xi \in \bd X - \{\sigma(+ \infty)\}$, see \cite[Lemma 1.7]{ballman_buyalo08}. In fact, the converse of can be made true once we add some conditions on the group action. The next propositions show that, given some conditions on the group action, there are ways to detect rank one elements in $G$ provided we have isolated points on the Tits boundary. 
	
	\begin{thm}[{\cite[Main Theorem]{ruane01}}]
		Let $G$ be group acting properly discontinuously, cocompactly by isometries on a $\cat(0) $ space $X$, and let $a, b$ be infinite order elements such that $d_T(\{a^{\pm \infty}\},\{b^{\pm \infty}\}) > \pi$, then the subgroup generated by $a$ and $b $ contains a free subgroup. In fact, there exists $N\geq 0 $ such that for all $n \geq N$, $a^n b^{-n }$ is a rank one isometry. 
	\end{thm}

	Given an action by isometries of a group $G$ on a $\cat(0)$ space $X$, the limit set $\Lambda$ of the action $G \curvearrowright X$ is the set of all points $ \xi \in \bd X$ such that there exists a sequence $(g_n) \in G $, and $x \in X$ for which $g_n x \rightarrow \xi $. 
	
	\begin{prop}[{\cite[Proposition 1]{ballman_buyalo08}}]
		Suppose that $\Lambda = \partial X$, and that for each $\xi \in \partial_T X$, there exists $\eta \in \partial_T X$ with $d_T(\xi, \eta) > \pi$. Then $G$ contains a rank one isometry. 
	\end{prop}

	\subsection{Stationary measures and boundary theory} \label{stationary section}
	
	In the study of random walks on a group $G$ acting on a metric space $X$, it is often very useful to use stationary measures on $X$. In the following, for $Y$ a measurable space, we denote $\prob(Y) $ the set of probability measures on $Y$. When a Polish group $G$ acts continuously on a topological probability space $Y$, with $\mu \in \prob(G)$ and $\nu \in \prob(Y)$, we define the convolution probability measure $\mu \ast \nu$ as the image of $\mu \times \nu $ under the action $G \times Y \rightarrow  Y $. In other words, for $f $ a bounded measurable function on $Y$, 
	
	\begin{equation*}
		\int_Y f(y)d(\mu \ast \nu)(y) = \int_G\int_Y f(g \cdot y ) d\mu(g) d\nu (y).
	\end{equation*}
	
	In our situation, $G$ is countable, so for $A$ any measurable set in $Y$, 
	
	\begin{equation*}
		\mu \ast \nu (A) = \sum_{g \in G} \mu(g) \nu (g^{-1}A). 
	\end{equation*}
	
	We will write $\mu_m = \mu^{\ast m}$ the $m$-th convolution power of $\mu$, where $G$ acts on itself by left translation $(g,h) \mapsto  gh$. 
	\begin{Def}
		A probability measure $\nu \in \prob(X) $ is $\mu$-stationary if $\mu \ast \nu = \nu $.	
	\end{Def}
	
	The Banach-Alaoglu Theorem implies that the set of measures on $\overline{X}$ is a weakly-$\ast$ compact space. The next result is a straightforward consequence of this fact. 
	
	\begin{thm}\label{existence}
		Let $G$ be a countable group acting by homeomorphisms on a compact metric space $Y$, and let $\mu \in \prob(G) $ a probability measure on $G$. Then there exists a $\mu$-stationary Borel probability measure $\nu \in \prob (Y)$ on $Y$. 
	\end{thm}
	
	\begin{rem}
		Since $X$ is a proper $\cat(0)$ space, $\overline{X}$ is a compact metrizable space and Theorem \ref{existence} states that there exists a probability measure $\nu$ on $\overline{X}$ that is $\mu$-stationary. 
	\end{rem}
	
	One of the reasons why we use stationary measures is given by the following Theorem, which is an important consequence of the martingale convergence Theorem and which goes back to Furstenberg \cite{furstenberg73}. 
	
	\begin{thm}[{\cite[Lemma 1.33]{furstenberg73}}]\label{furstenberg73}
		Let $G$ be a discrete group, $\mu \in \prob(G)$ and $(n, \omega) \in \mathbb{N} \times \Omega \mapsto Z_n(\omega)$ be the random walk on $G$  associated to the measure $\mu$. Let $Y$ be a locally compact, $\sigma$-compact metric space on which $G$ acts by isometries, and let $\nu$ be a $\mu$-stationary measure on $Y$. Then, for $\mathbb{P}$-almost every $\omega \in \Omega$, there exists $\nu_\omega \in \prob(Y)$ such that $Z_n (\omega)\nu \rightarrow \nu_\omega $ in the weak-$\ast$ topology. Moreover, for all $g \in G$, $Z_n (\omega)g\nu \rightarrow \nu_\omega $ in the weak-$\ast$ topology, and $\nu = \int_{\Omega} \nu_\omega d \mathbb{P}(\omega)$. 
	\end{thm}

	Let us give a brief overview of boundary theory. For more details, one can study \cite{kaimanovitch00} and \cite{furman02}. 
	We define by "shift map" the application defined by 
	\begin{equation*}
		S : (\omega_0, \omega_1,\dots) \in \Omega \mapsto (\omega_0 \omega_1, \omega_2, \dots).
	\end{equation*}
	If we define by $ f  :\Omega \rightarrow G$ the application $f(\omega) = \omega_1$, see random walk $(Z_n)$ can be written as 
	\begin{equation*}
		(n, \omega) \in \mathbb{N} \times \Omega \mapsto Z_n(\omega) = f(\omega) f(S\omega) \dots f(S^{n-1} \omega). 
	\end{equation*}
	
	Given a measure $\mu $ on a group and a measurable $G$-action on a metric space $ M $ endowed with a probability $\nu$, we say that $ (M, \nu) $ is a $(G, \mu)$-space if the measure $\nu$ is $\mu$-stationary. In that case, Theorem \ref{furstenberg73} states that there exists a limit measure $\nu_\omega = \lim_{n \rightarrow \infty } Z_n(\omega) \nu$. 
	
	\begin{Def}
		A $(G, \mu)$-space $(M, \nu)$ is a $(G,\mu)$-boundary if for $\mathbb{P}$-almost every $\omega\in \Omega$, $\nu_\omega$ is a Dirac measure. 
	\end{Def}
	The study of $(G, \mu)$-boundaries has strong connections with the existence of harmonic functions and Poisson transforms on a group, but these results will be omitted here. We refer to \cite{furman02} for more informations on these subjects. 
	
	It is straightforward to see that any $G$-equivariant factor $(M', \nu')$ of a $(G, \mu)$-boundary is still a $(G, \mu)$-boundary. In fact, the following Theorem states that there exists a pair $(B, \nu_B)$ which is maximal and universal among $(G, \mu)$-boundaries. 
	
	\begin{thm}[{\cite[Theorem 10.1]{furstenberg73}}]\label{poisson boundary}
		Given a locally compact group $G$ with admissible probability measure $\mu$, there exists a maximal $(G, \mu)$-boundary, called the Poisson-Furstenberg boundary $(B, \nu_B)$ of $(G, \mu)$, which is uniquely characterized by the following property:
		\begin{description}
			\item [Universality]For every measurable $(G, \mu)$-boundary $(M, \nu)$, there is a $G$-equivariant measurable quotient map $p : (B, \nu_B) \rightarrow (M, \nu)$, uniquely defined up to $\nu_B$-null sets. 
		\end{description}
	\end{thm}
	
	A construction of the Poisson boundary can be described as follows.	The Poisson boundary of the measure $\mu$ is the space $B$ of ergodic components of the action of $S$ on $(\Omega, \haar \otimes \mu^{\mathbb{N^\ast}})$. It is a measured space equipped by the pushforward $\nu_B$ of the measure $\mathbb{P} $ by the natural projection $\Omega \rightarrow B$. 
	The space $(B, \nu_B)$ is a a $(G, \mu)$-space on which $G$ acts ergodically. There is a lot to say about the action of $G$ on $B$, especially concerning $G$-equivariant boundary maps, and the interested reader may read \cite{bader_furman14} for recent developments in this direction, where it is proven that the action $G \curvearrowright B$ is isometrically ergodic, which is an enhanced version of ergodicity.

	\section{Uniqueness of the stationary measure}\label{uniqueness measure section}
	
	From now on, let $G$ be a countable group and $G \curvearrowright X$ a non-elementary action by isometries on a proper $\cat (0)$ space $X$. Let $\mu \in \prob(G) $ be an admissible probability measure on $G$, and assume that $G $ contains a rank one element. Theorem~\ref{existence} gives the existence of a $\mu$-stationary measure $\nu \in \prob(\overline{X})$ on $\overline{X}$. The goal of this section is to show that $\nu$ is the unique $\mu$-stationary measure on $\overline{X}$. In order to do so, we show that the measures $\nu_\omega $ given by the Theorem \ref{furstenberg73} are in fact Dirac measures $\delta_{z(\omega)}$, and that they do not depend on $\nu $. 
	
	\subsection{Dynamics on the Tits boundary}
	
	Let $x \in X$ be a basepoint. We start by showing that almost surely, a subsequence of $(Z_n(\omega)x)$ goes to infinity. 
	
	\begin{lem}\label{subsequence}
		For all $x \in X$,  $(Z_n(\omega) x)_n$ is $\mathbb{P}$-almost surely unbounded.  
	\end{lem}
	
	\begin{rem}\label{subsequence rem}
		Since $\overline{X}$ is compact, a straightforward consequence is that $\mathbb{P}$-almost surely, there exist a subsequence $\phi (n)$ (depending on $\omega$) and $z^{+} (\omega), z^{-}(\omega) \in \bd X$ such that $(Z_{\phi(n)}(\omega) x)_n$ converges to $z^{+} (\omega)$, and $(Z^{-1}_{\phi(n)}(\omega) x)_n$ converges to $z^{-}(\omega)$.
	\end{rem}
	
	\begin{proof}
		Let $K$ be a compact subset of $X$, and let $D$ be its diameter. By hypothesis, there exists $g$ a rank one element in the group, of translation length $\tau(g) := l > 0$. Hence there exists $k \in \mathbb{N}$ such that $\tau(g^k) = k l > D$. In particular, if $Z_n(\omega) x \in K$, then $Z_n(\omega) g^k x \notin K$. Since $\mu $ is admissible, there exists $m \in \mathbb{N}$ such that $\mu_m(g^k) =: a >0$. Then for $n \in \mathbb{N}$, 
		\begin{equation*}
			\mathbb{P}(Z_{n+m}(\omega)x \in K \, | \,  Z_n(\omega)x \in K) < 1-a. 
		\end{equation*}
		Then, $\mathbb{P}$-almost surely, there exists $n_0 $ such that $Z_{n_0}(\omega) x \notin K $. 
		
		Let us take an increasing sequence of compacts $(K_p)_p$ such that $\bigcup K_p = X$. For all $p \in \mathbb{N}$, $\mathbb{P}$-almost surely there exists $n_p \in \mathbb{N}$ such that $Z_{n_p}(\omega) x \notin K_p $. Then the subsequence $(Z_{n_p}(\omega) x)_p$ escapes every compact of $X$. Since $G$ acts by isometries, $(Z^{-1}_{n_p}(\omega) x)_p$ also escapes every compact of $X$.
	\end{proof}
	
	The following Theorem, due to P. Papasoglu and E. Swenson in \cite{papasoglu_swenson09} is a key ingredient in our proof. 
	
	\begin{thm}[{\cite[Lemma 19]{papasoglu_swenson09}}]\label{PS}
		Let $X$ be a proper $\cat (0)$ space, and $G \curvearrowright X$ an action by isometries. Let $x \in X$, $\theta \in [0, \pi]$ and $(g_n) \subseteq G$ be a sequence of isometries for which there exists $x \in X $ such that $g_n(x) \rightarrow \xi \in \bd X $ and $ g^{-1}_n(x) \rightarrow \eta \in \bd X$. Then for all compact subset $K \subseteq \bd X -  B_T(\eta, \theta)$ and for all open subset $U$ such that $B_T(\xi, \pi - \theta)\subseteq U$, there exists $n_0$ such that for all $n \geq n_0$, $g_n(K) \subseteq U$. 
	\end{thm}
	
	We want to apply this theorem in order to prove that the limit measures given by Theorem \ref{furstenberg73} are Dirac measures. First, we start by a technical Lemma. 
	
	\begin{lem}\label{separation}
		Let $g$ be a rank one isometry in $G$, with fixed points $g^{+}, \, g^{-} \in \bd X$ respectively attractive and repulsive. Then there exists $U^{+}, U^{-} \subseteq \bd X$ neighbourhoods of $g^{+}, \, g^{-}$ respectively such that for all $ \xi \in \bd X$, \begin{equation*}
			B_T(\xi, \pi)  \cap U^{-} \neq \emptyset  \Rightarrow B_T(\xi, \pi)  \cap U^{+} = \emptyset,
		\end{equation*}
		and \begin{equation}
			B_T(\xi, \pi)  \cap U^{+} \neq \emptyset  \Rightarrow B_T(\xi, \pi)  \cap U^{-} = \emptyset.
		\end{equation} In other words, we can find neighbourhoods of the fixed points of $g$ small enough so that the Tits ball of radius $\pi $ around any point $\xi \in \bd X$ do not intersect both neighbourhoods simultaneously.
	\end{lem}
	
	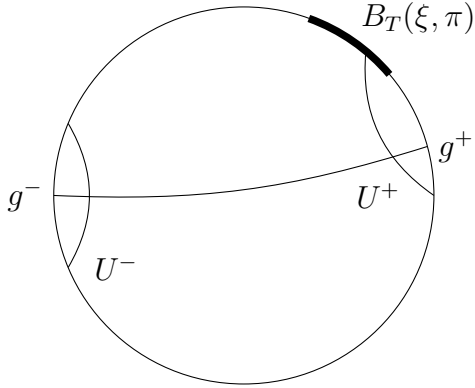
\begin{figure}
		\centering
		\begin{center}
			\begin{tikzpicture}[scale=0.5]
				\draw (0,0) circle (5cm); 
				\draw[line width =1mm] (3.83, 3.21) arc (40 : 70 : 5cm) node[right=5mm] {$B_T(\xi, \pi)$};
				\draw (3.21, 3.83) to[bend right] (5,0) node[left=3mm] {$U^{+}$} ;
				\draw (-4.62, 1.91) to[bend left] (-4.62, -1.91) node[right=2mm] {$U^{-}$} ;
				\draw  (-5, 0) to[bend right = 10] (4.83, 1.3) ; 
				\draw (4.83, 1.3) node[above, right] {$g^{+}$}; 
				\draw (-5,0) node[above, left]{$g^{-}$};
			\end{tikzpicture}
		\end{center}
		\caption{Illustration of Proposition \ref{separation}.}
	\end{figure}
	
	\begin{proof}
		By contradiction, assume there exist decreasing sequences of neighbourhoods $\{ U^{+}_n \} $ and $\{ U^{-}_n \} $  of respectively $g^{+}$ and $g^{-}$, such that $\cap_n U^{+}_n = g^{+}$ and $\cap_n U^{-}_n = g^{-}$, and a sequence of points $(\xi_n) \subseteq \bd X$ such that for all $n \in \mathbb{N}$, $B_T(\xi_n, \pi) \cap U^{-}_n \neq \emptyset $ and $B_T(\xi_n, \pi) \cap U^{+}_n \neq \emptyset $. Notice that due to Theorem \ref{tits}, since $g^{-}$ is a fixed point of a rank one isometry, $d_T(g^{-}, \eta) = + \infty$ for all $\eta \in \bd X - \{ g^{-}\} $. 
		
		For all $n \in \mathbb{N}$, take $z_n \in B_T(\xi_n, \pi) \cap U^{-}_n$. By hypothesis, $z_n \rightarrow g^{-}$. Since $\bd X$ is compact, then passing to a subsequence, $\xi_n \rightarrow \xi \in \bd X$ in the visual topology. By lower semicontinuity of the Tits metric (Theorem \ref{tits}) , $\liminf d_T (\xi_n, z_n) \geq d_T(g^{-}, \xi )$, hence $d_T(g^{-}, \xi) \leq \pi$, thus $\xi = g^{-}$. The same argument with $z_n \in B_T(\xi_n, \pi) \cap U^{+}_n$ gives $\xi = g^{+}$, a contradiction. 
	\end{proof}

	The next step is to show that the measure of this ball is actually zero. Let us begin by showing that $\nu$ is non-atomic, that is, for all $\xi \in \overline{X}$, $\nu(\xi) = 0 $. The next Lemma follows standard ideas. 
	
	\begin{lem}\label{non atomic}
		Let $(G, \mu)$ be a discrete group acting by isometries on a metric space $X$, and let $\nu$ be a $\mu$-stationary measure on $X$. If the action $G \curvearrowright X$ does not have finite orbits, then $\nu$ is non-atomic. 
	\end{lem}
	
	\begin{proof}
		Let us assume that there exists an atom for $\nu$. Let $m :=  \max\{ \nu(x) : x \in \overline{X}\}$ and $X_m = \{ x \in \overline{X} : \nu (x) = m\}$. The set $X_m $ is non-empty by hypothesis, and finite because $\nu(\overline{X}) = 1$.  
		Let $x \in X_m$. Since $\nu $ is $\mu$-stationary,  $\mu \ast \nu (x ) = \nu (x)$, hence
		\begin{equation}
			\sum_{g \in G} \mu(g) \nu(g^{-1} x ) = m \nonumber. 
		\end{equation}
		Then, for all $g \in G$, $\nu(g^{-1} x) = m$. The set $X_m $ is $G$-invariant, finite and non-empty, which is in contradiction with the fact that the action does not have finite orbits. 
	\end{proof}
	
	\begin{rem}\label{r1}
		The group $G$ acts on $(\partial_T X, d_T)$ by isometries, hence for all $\xi \in \bd X$, $f \in G$, $fB_T(\xi, \pi) = B_T(f\xi, \pi)$. 
	\end{rem}
	
	\subsection{$B_T(\xi, \pi)$ is a null set}
	
	In this section, we show that $\nu(B_T(\xi, \pi))=0$. In order to do this, we use a Lemma from J. Maher and G. Tiozzo (\cite{maher_tiozzo18}), and North-South dynamics on the boundary. 
	
	\begin{lem}[{\cite[Lemma 4.5]{maher_tiozzo18}}]\label{mt}
		Let $G$ be a discrete group acting by homeomorphisms on a metric space $M$, and let $\mu \in \prob(G)$ whose support generates $G$ as a semigroup. Let $\nu $ be a $\mu$-stationary measure on $M$. Let $Y \subseteq M$,and assume that there exists a sequence of positive numbers $(\varepsilon_n)_{n \in \mathbb{N}}$ such that for all $f \in G$, there exists a sequence $(g_n)_{n \in \mathbb{N}}$ such that the translates $fY, g^{-1}_1 f Y, g^{-1}_2 f Y, \dots $ are all disjoint and for all $g_n$, there exists $m \in \mathbb{N}$ such that $\mu_m(g_n) \geq \varepsilon _n$. Then $\nu(Y)= 0$. 
	\end{lem}
	
	We can then apply the Lemma \ref{mt} to obtain the following result: 
	
	\begin{prop}\label{zero measure}
		With the previous notations, for all $\xi \in \bd X$, $\nu (B_T( \xi, \pi)) = 0$. 
	\end{prop}
	\begin{proof}

	By assumption, the action of $G$ on $X$ is non-elementary and there exists a rank one isometry. By Proposition \ref{non elem caprace fuj}, there exist two rank one isometries $g$ and $h$ whose fixed points are mutually disjoint. If we repeat the argument given in Lemma \ref{separation}, we obtain that there exists $U^{+}_g, U^{-}_g, U^{+}_h , U^{-}_h $ neighbourhoods of $g^{+}, g^{-}, h^{+}, h^{-}$ respectively such that for all $\xi \in \bd X, \, B_T(\xi, \pi) \cap U^{+}_g \neq \emptyset \Rightarrow B_T(\xi, \pi) \cap U = \emptyset$, for $U = U^{-}_g, U^{+}_h , U^{-}_h $ and $B_T(\xi, \pi) \cap U^{-}_g \neq \emptyset \Rightarrow B_T(\xi, \pi) \cap U = \emptyset$, for $U = U^{+}_g, U^{+}_h , U^{-}_h $.
	\newline

	Let us apply Lemma \ref{hamenstadt} to rank one isometries $g$ and $h$ with distinct fixed points. There exists $k_0$ such that for all $k \geq k_0$, $g^{k} (\bd X  - U^{-}_g) \subseteq U^{+}_g$, $g^{-k} (\bd X  - U^{+}_g) \subseteq U^{-}_g$, $h^{k} (\bd X  - U^{-}_h) \subseteq U^{+}_h$ and $h^{-k} (\bd X  - U^{+}_h) \subseteq U^{-}_h$. 
	
	Let now $\xi \in \bd X$ be a boundary point, $f \in G$ an isometry and write $Y := B_T(\xi, \pi)$. We have shown in Remark \ref{r1} that $fY = B_T(f \xi, \pi)$. We are looking for a sequence of elements $(g_n)_n$ of $G$ such that the conditions in Lemma~\ref{mt} are verified. There are three cases possible: 
	\newline
	
	\textbf{\underline{First case}}:
	If $fY \cap \ugm \neq \emptyset$, then $fY \cap \uhm = \emptyset$ and $fY \cap \uhp = \emptyset$. Hence, for all $k \geq k_0$, $h^{k} f Y \subseteq \uhp $. The translates 
	\begin{equation*}
		\{ fY, h^{ k_0} f Y, h^{ 2k_0} f Y, \dots h^{ n  k_0} f Y\dots \}
	\end{equation*}
	are all disjoint. Indeed, $f Y \cap \uhp = \emptyset $ by hypothesis and if there exists $p \in \mathbb{N}$ such that $h^{ (n + p)k_0} f Y \cap h^{ n k_0} f Y \neq \emptyset$. Then $h^{pk_0} fY \cap fY \neq \emptyset$, which is impossible because $h^{pk_0} fY \subseteq \uhp $ and $fY \cap \uhp = \emptyset$. 
	\newline
	
	\textbf{\underline{Second case}}:
	If $fY \cap \ugp \neq \emptyset$, then the translates 
	\begin{equation*}
		\{ fY, h^{ k_0} f Y, h^{2k_0 }fY, \dots h^{ n  k_0} f Y\dots \}
	\end{equation*} are all disjoint for the same reasons.
	\newline
	
	\textbf{\underline{Third case}}: 
	If $fY \cap \ugm = \emptyset$ and $fY \cap \ugp = \emptyset$, then the same argument shows that the translates 
	\begin{equation*}
		\{ fY, g^{ k_0} f Y, g^{ 2k_0} f Y, \dots g^{ n k_0} f Y\dots \}
	\end{equation*}
	are all disjoint. 
	\newline
	
	\begin{figure}
		\centering
		\begin{center}
			\begin{tikzpicture}[scale=0.5]
				\draw (0,0) circle (5cm); 
				\draw[line width =1mm] (3.83, 3.21) arc (40 : 70 : 5cm) node[right=5mm] {$B_T(f\xi, \pi)$};
				\draw (3.21, 3.83) to[bend right] (5,0) node[left=3mm] {$U_h^{+}$} ;
				\draw (-4.62, 1.91) to[bend left] (-4.62, -1.91) node[right=2mm] {$U_g^{-}$} ;
				\draw  (-5, 0) to[bend right = 10] (4.83, 1.3) ; 
				\draw  (0,-5) to[bend right = 10] (-2.5, 4.33) ; 
				\draw  (1.29,-4.83) to[bend right = 40] (-1.29,-4.83) node[above=3mm] {$U_h^{-}$}; 
				\draw  (-3.53, 3.53) to[bend right = 40] (-1.29,4.83) node[below=3mm] {$U_h^{+}$};
				\draw (4.83, 1.3) node[above, right] {$g^{+}$}; 
				\draw (-5,0) node[above, left] {$g^{-}$};
				\draw (-2.5, 4.33) node[above, left] {$h^{+}$} ; 
				\draw (0,-5) node[below]{$h^{-}$} ;
			\end{tikzpicture}
		\end{center}
		\caption{Illustration of Lemma \ref{zero measure}.}
	\end{figure}
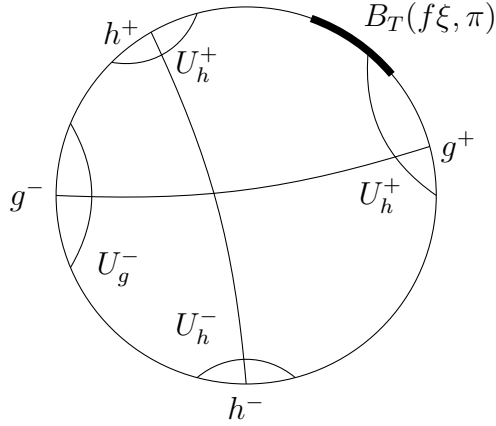
	
	Now, since the measure $\mu $ is admissible, for all $n \in \mathbb{N}$, there exists $ m \in \mathbb{N}$ (depending on $n$) such that $\mu_m (h^{-nk_0}) > 0$ . Similarly, for all $n\in \mathbb{N}$, there also exists $m' \in \mathbb{N} $ such that $\mu_{m'}(g^{-nk_0})  > 0 $. Consider the sequence $(\varepsilon_n)_{n \in \mathbb{N}}$ defined by $\varepsilon_n := \min \{\mu_m (h^{-nk_0}), \mu_{m'}(g^{nk_0})\}$. It is a sequence of positive numbers, and it does not depend on $f$. Finally, for all $f \in G$, we have found a sequence $(g_n)$ such that the translates $\{ fY, \dots, g^{-1}_n f Y\}$ are all disjoint and such that for all $n$, there exists $m \in \mathbb{N}$ such that $\mu_m(g_n) \geq \varepsilon_n$. The proposition follows from Lemma \ref{mt}. 
	
	\end{proof}

	We can now prove the main result of this section.

	\begin{thm}\label{unicite}	
		Let $G$ be a discrete group and $G \curvearrowright X$ a non-elementary action by isometries on a proper $\cat (0)$ space $X$. Let $\mu \in \prob(G) $ be an admissible probability measure on $G$, and assume that $G $ contains a rank one element. Then there exists a unique $\mu$-stationary measure $\nu \in \prob(X)$. 
	\end{thm}
	
	\begin{proof}
		By Lemma \ref{subsequence} and Remark \ref{subsequence rem}, for all $x \in X$, there almost surely exists a subsequence $(Z_{\phi (n)}(\omega)x)_n $ of $(Z_n(\omega) x)_n$ and $z^{+} (\omega), z^{-}(\omega) \in \bd X$ such that $(Z_{\phi(n)}(\omega) x)_n$ converges to $z^{+} (\omega)$, and $(Z^{-1}_{\phi(n)}(\omega) x)_n$ converges to $z^{-}(\omega)$. By Theorem \ref{PS}, for all $K \subseteq \bd X - B_T( z^{-}(\omega), \pi)$, and for all $U$ neighbourhood of $z^{+}(\omega)$, there exists $n_0$ such that $Z_{\phi(n)}(\omega)K \subseteq U$ for all $n \geq n_0$. By Proposition \ref{zero measure}, $\nu (B_T( z^{-}(\omega), \pi)) = 0$ hence for all measurable $A$ of $\bd X$,  $Z_{\phi(n)}(\omega)\nu(A)$ converges to $1$ if $z^{+}(\omega) \in A$ and to $0$ otherwise. In other words, $Z_{\phi(n)}(\omega) \nu \rightarrow \delta_{z^{+}(\omega)}$ in the weak-$\ast$ topology. By Theorem \ref{furstenberg73}, $Z_n(\omega) \rightarrow \nu_\omega$ in the weak-$\ast$ topology, so $\nu_\omega = \delta_{z^{+}(\omega)}$ by uniqueness of the limit. Since $\nu = \int_{\Omega} \delta_{z^{+}(\omega)} d \mathbb{P}(\omega)$ and $z^{+}(\omega)$ does not depend upon $\nu$, the measure $\nu \in \prob(X) $ is the unique $\mu$-stationary measure on $\overline{X}$.
	\end{proof}
	
	\section{Convergence of the random walk} \label{convergence random walk}

	The goal of this section is to show that for all $x \in X$, $Z_n(\omega)x \rightarrow z^{+}(\omega)$ $\mathbb{P}$-almost surely. Note that since $G$ acts by isometries, if $ (Z_n(\omega)x) $ converges to $\xi \in \bd X$ for some $x \in X$, then $(Z_n(\omega)x') $ also converges to $\xi $ for all $x' \in X$. 
	
	The following is known as Portmanteau Theorem, and is a classical result in measure theory. 
	
	\begin{prop}
		Let $Y$ be a metric space, $P_n$ a sequence of probability measures on $Y$, and $P$ a probability measure on $Y$. Then the following are equivalent: 
		\begin{itemize}
			\item $P_n \rightarrow_n P$  in the weak$-\ast$ topology; 
			\item $\underset{n \rightarrow \infty}{\liminf} \ P_n(O) \geq P(O)$ for every open set $O \subseteq Y$. 
		\end{itemize}
	\end{prop}
	
	\begin{cor}\label{portmanteau}
		Let $O$ be an open neighbourhood of $z^{+}(\omega)$ (for the visual topology). Then \begin{equation}
			\liminf_{n \rightarrow \infty} \nu(Z^{-1}_n(\omega) (O)) = \delta_{z^{+}(\omega)}(O) = 1. 
		\end{equation}
	\end{cor}
	
	The next result was proven by W. Ballmann in \cite{ballman95}, and will be fundamental in the sequel. 
	
	\begin{lem}[{\cite[Lemma III.3.1]{ballman95}}]\label{ballm}
		Let $\sigma : \mathbb{R} \rightarrow X$ a bi-infinite geodesic in $X$ that does not bound a flat half strip of width $R >0$, with endpoints $\sigmam $ and $\sigmap$ in $\bd X$. Then there exist neighbourhoods $U$, $V$ of  $\sigmam$ and $ \sigmap$ respectively in $\overline{X}$ such that for all $\xi \in U$ and $\eta \in V$, there is a geodesic $\sigma'$ in $X$ from $\xi $ to $\eta$, and for any such geodesic, we have $d(\sigma(0), \sigma') < R$. In particular, $\sigma'$ does not bound a flat strip of width $2R$. 
	\end{lem}

	We recall that we have proven in section \ref{stationary section} that $\nu$ is the unique stationary measure on $\overline{X}$, that  $\mathbb{P}$-almost surely, $Z_n (\omega)\nu \rightarrow \delta_{z^{+}(\omega)}$ for some $z^{+}(\omega) \in \bd X$, and that $\nu $ is distributed as $\nu = \int_{\Omega} \delta_{z^{+}(\omega)} d \mathbb{P}(\omega)$. In other words, for all open set $U$ in $\overline{X} $, $\nu(U) =  \mathbb{P} (\omega \in \Omega \, : \, z^{+}(\omega) \in U)$. We also know from Lemma \ref{non atomic} that $\nu$ is non atomic. 
	
	\begin{rem}\label{support}
		The \textit{support} of a measure $m$ on a topological space $Y$ is the smallest closed set $C$ such that $m(Y \setminus C)= 0$. In other words $y \in \supp(m)$ if and only if for all $U $ open containing $y $, $m(U) >0$. From what we have obtained in Section \ref{stationary section}, $\supp(\nu)$ is infinite, and for each $z \in \supp(\nu)$, $\nu(B_T(z, \pi))= 0 $. In other words, any two points of the support of $\nu $ are almost surely joined by a rank one geodesic. 
	\end{rem}

	Using Proposition \ref{continuite proj}, one can now prove that the random walk goes to infinity almost surely. 
	
	\begin{lem}\label{unbdd}
		Let $x \in X$ a basepoint. Then $d(x, Z_n(\omega)x) \rightarrow \infty $ almost surely. 
	\end{lem}
	
	\begin{proof}
		Let $z_1$ and $z_2$ be two distinct points of the support of $\nu$. By Remark \ref{support}, we can take $z_1$ and $z_2$ to be joined by a rank one geodesic $\sigma$. Let us suppose without loss of generality that $\sigma(- \infty) = z_1$ and $\sigmap= z_2$. Recall that by Remark \ref{flat strip}, there exists $R > 0 $ large enough so that $\sigma $ does not bound a flat strip of width $R$. Since $G$ acts by isometries on $X$, what we want to show does not depend on the basepoint and we can take $x~:=~\sigma(0)$. Let us assume by contradiction that $(Z_n(\omega)x)_n$ admits a bounded subsequence. 
		
		Because $X$ is proper, there exists $y \in X$ and a subsequence $(\phi(n))_n$ such that $Z_{\phi(n)}(\omega) x \rightarrow y \in X$. In particular, there exists $n_0$ such that for all $n \geq n_0$, $d(Z_{\phi(n)}(\omega) x , y) \leq 1$. 
		
		Due to Lemma \ref{portmanteau}, for every open neighbourhood $O$ of $z^{+}(\omega)$ and every $\varepsilon>0$, there exists $N \in \mathbb{N}$ such that for all $n~\geq~N$, $\nu(Z^{-1}_n(\omega) O) \geq 1 - \varepsilon$. 
		
		Now define $U$, $V$ to be the open neighbourhoods of $\sigmap$ and $\sigma(-\infty)$ respectively given by Lemma \ref{ballm}. Since $z_1 $ belongs to the support of $\nu$, $U$ has non-zero $\nu$-measure, thus in particular there exists $n_1\geq n_0$ such that for all $n \geq n_1$, $Z_n(\omega)U \cap O \neq \emptyset $. Repeating the same argument with $V$, there exists $n_2 \geq n_1$ so that for all $n \geq n_2$, $Z_n(\omega)U \cap O \neq \emptyset $ and $Z_n(\omega)V \cap O \neq \emptyset $. 
		
		Now fix $r, \varepsilon > 0 $, with $r > \varepsilon$, and let $R'= R'(r, R+1, \varepsilon) $ given by Proposition \ref{continuite proj}. Observe that the set $O := U(y, z^{+}(\omega), R', \varepsilon/3)$ is an open neighbourhood of $z^{+}(\omega)$, so by the previous argument, there exists $n_2 \in \mathbb{N}$ and $(\xi, \eta) \in U \times V$ such that $Z_{\phi(n_2)}(\omega)\xi$ and $Z_{\phi(n_2)}(\omega)\eta$ both belong to $O$. Now by Lemma~\ref{ballm}, there exists a geodesic line $\sigma_{\xi, \eta}$ in $X$ joining $\xi $ and $\eta $ such that $d(x, \sigma_{\xi, \eta})\leq R$. Let $x'$ be the projection of $x$ on $\sigma_{\xi, \eta}$, so that $d(x,x') \leq R$. 
		
		Then for all $n \in \mathbb{N}$, 
		\begin{eqnarray}
			d(Z_{\phi(n)}(\omega) x', \, y) &\leq& d( Z_{\phi(n)}(\omega) x' , Z_{\phi(n)}(\omega) x ) + d (Z_{\phi(n)}(\omega) x, y) \nonumber \\
			& \leq &  d(  x' , x ) + d (Z_{\phi(n)}(\omega) x, y) \nonumber \\
			& \leq &  R + d (Z_{\phi(n)}(\omega) x, y) \nonumber.
		\end{eqnarray}
		In particular, applying this equality for $n= n_2$ yields $d(Z_{\phi(n_2)}(\omega) x', \, y) \leq R+1$. From now on, denote $Z_{\phi(n_2)}(\omega) x'$ by $y'$, and by $p_{r}$ the closest point projection $p_{r} : \overline{X} \rightarrow \overline{B}(y', r)$.	Due to Proposition~\ref{continuite proj}, and because we have chosen $R'$ accordingly,
		\begin{equation*}
			U(y, z^{+}(\omega), R', \varepsilon/3) \subseteq U(y', z^{+}(\omega), r, \varepsilon).
		\end{equation*} 
		In particular, since we have defined $\xi $ and $\eta $ such that $Z_{\phi(n_2)}(\omega)\xi$ and $Z_{\phi(n_2)}(\omega)\eta$ belong to $U(y, z^{+}(\omega), R', \varepsilon/3)$, it implies that $Z_{\phi(n_2)}(\omega)\xi$ and $Z_{\phi(n_2)}(\omega)\eta$ both belong to $U(y', z^{+}(\omega), r, \varepsilon)$. However, there is a geodesic line from $Z_{\phi(n_2)}(\omega)\xi$ to $Z_{\phi(n_2)}(\omega)\eta$ passing through $y' = Z_{\phi(n_2)}(\omega) x'$, so that 
		\begin{equation*}
			d(p_{r}(Z_{\phi(n_2)}(\omega)\xi), p_{r}(Z_{\phi(n_2)}(\omega)\eta)) = 2r.
		\end{equation*}
		Now $\xi $ and $\eta$ both belong to $U(y', z^{+}(\omega), r, \varepsilon)$, which means that 
		\begin{equation*}
			d(p_{r}(Z_{\phi(n_2)}(\omega)\xi), p_{r}(z^{+}(\omega))) < \varepsilon,
		\end{equation*}
		and 
		\begin{equation*}
			d(p_{r}(Z_{\phi(n_2)}(\omega)\eta), p_{r}(z^{+}(\omega))) < \varepsilon. 
		\end{equation*}
		
		Now by the triangular inequality, 
		\begin{equation*}
			d(p_{r}(Z_{\phi(n_2)}(\omega)\xi), p_{r}(Z_{\phi(n_2)}(\omega)\eta)) < 2 \varepsilon,
		\end{equation*}
		a contradiction with the fact that $\epsilon < r$. See Figure \ref{lemme 4.4}.
	\end{proof}
	
	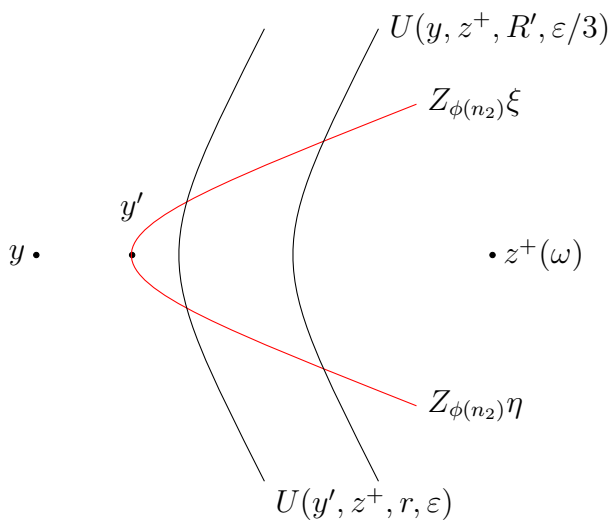
\begin{figure}[h]\label{lemme 4.4}
		\centering
		\begin{center}
			\begin{tikzpicture}[scale=1]
				\filldraw[black] (0,0) circle (1pt) node[left] {$y$} ; 
				\filldraw[black] (6,0) circle (1pt) node[right] {$z^{+}(\omega)$} ;
				\draw  (5,2) node[right] {$Z_{\phi(n_2)} \xi$} ;
				\draw  (5,-2) node[right] {$Z_{\phi(n_2)} \eta$} ;
				\filldraw[black] (1.26,0) circle (1pt) node[above= 3mm] {$y'$} ; 
				\draw[red] (5,2) .. controls (0,0) .. (5,-2);
				\draw (3,3) .. controls (1.5,0) .. (3,-3) node[below=3mm,right] {$U(y',z^{+}, r, \varepsilon)$};
				\draw (4.5,-3) .. controls (3,0) .. (4.5,3) node[above,right] {$U(y,z^{+}, R', \varepsilon/3)$};
			\end{tikzpicture}
		\end{center}
		\caption{$Z_{\phi(n_1)} \xi$ and $Z_{\phi(n_1)} \eta$ can not belong to $U(y',z^{+}, r, \varepsilon)$.}
	\end{figure}

	Now we can prove the convergence to the boundary. 
	
	\begin{thm}\label{convergence}
		Let $G$ be a discrete group and $G \curvearrowright X$ a non-elementary action by isometries on a proper $\cat (0)$ space $X$. Let $\mu \in \prob(G) $ be an admissible probability measure on $G$, and assume that $G $ contains a rank one element. Then for every $x \in X$, and for $\mathbb{P}$-almost every $\omega \in \Omega$, the random walk $(Z_n (\omega) x)_n $ converges to a boundary point $z^{+}(\omega) \in \bd X$. Moreover, $z^{+}(\omega)$ is distributed according to $\nu $. 
	\end{thm}
	
	\begin{proof}
		Let $x \in X$. Because of Lemma \ref{unbdd} the random walk $(Z_n (\omega) x)_n$ goes to infinity, so it is enough to show that there is no accumulation point of $(Z_n (\omega) x)_n $ in $\bd X$ other than the boundary point $z^{+} (\omega)$ given by Proposition \ref{subsequence}. Assume that for a given subsequence, $Z_{\phi(n)}(\omega) x \rightarrow \xi$, with $\xi \in \bd X$. Then we can apply the results in the first section and the Theorem \ref{PS} to get that $Z_{\phi(n)}(\omega) \nu \rightarrow \delta_\xi$ in the weak-$\ast$ topology. By Theorem \ref{furstenberg73}, we have $Z_n (\omega) \nu \rightarrow \delta_{z^{+}(\omega)}$, so $z^{+}(\omega) = \xi$ by uniqueness. 
	\end{proof}
	
	Now, Proposition \ref{zero measure} combined with Theorem \ref{tits} allows to state a geometric result which can be of independent interest.
	
	\begin{cor}\label{limit points are rank one}
		Let $\xi \in \bd X$ be a limit of the random walk $(Z_n x)_n$. Then for $\nu$-almost every point $\eta \in \bd X$, there exists a rank one geodesic joining $\xi $ to $\eta$. 
	\end{cor}

	\section{Positivity of the drift}\label{drift section}
	
	\subsection{Proof of Theorem \ref{drift thm}}
	
	Now that we know that the random walk converges to the boundary, we can wonder "at which speed" it converges. The goal of this section is to show that this speed is linear. The strategy is classical: it was initiated by Guivarc'h and Raugi for random walks on Lie groups \cite{guivarch_raugi85} and later it was used later for the study of free groups by Ledrappier \cite{ledrappier01}. This type of results can be understood as a generalised version of a Law of Large Numbers for a given random walk in some metric space. These questions have been extensively studied by Benoist and Quint in \cite{benoist_quint}, who have also proven a Central Limit Theorem for random walks on Gromov-hyperbolic groups, see \cite{benoist_quint16}.
	
	Let $G$ be a discrete group and $G \curvearrowright X$ a non-elementary action by isometries on a proper $\cat (0)$ space $X$. Let $\mu \in \prob(G) $ be an admissible probability measure on $G$. As a consequence of Kingman subadditive Theorem (see for example \cite[Corollary 4.3]{karlsson_margulis}), there is a constant $\lambda$ such that for $\mathbb{P}$-almost every sample path $(Z_n(\omega)x)_n$ we have 
	\begin{equation}
		\lim_{n \rightarrow \infty} \frac{1}{n} d(Z_n(\omega) x, x) = \lambda.
	\end{equation}
	
	The aforementioned constant $\lambda$ is referred to as the \textit{drift} of the random walk. We prove that if we assume that the probability measure $\mu $ has finite first moment, i.e. $\sum_G \mu(g) d(gx, x) < \infty$, the drift can be written by $ \lim_{n \rightarrow \infty} \frac{1}{n} d(Z_n(\omega) x, x) = \lambda $ and is positive. More precisely, we establish the following result: 
	
	\begin{thm}\label{drift}
		Let $G$ be a discrete group and $G \curvearrowright X$ a non-elementary action by isometries on a proper $\cat (0)$ space $X$. Let $\mu \in \prob(G) $ be an admissible probability measure on $G$ with finite first moment, and assume that $G $ contains a rank one element. Let $x \in X$ be a basepoint of the random walk. Then the drift $\lambda$ is almost surely positive: 
		\begin{equation}
			\lim_{n \rightarrow \infty} \frac{1}{n} d(Z_n(\omega) x, x) = \lambda > 0. 
		\end{equation}
	\end{thm}
	
	From now on, we denote by $\nu$ the unique $\mu$-stationary measure on $X$ given by Theorem \ref{unicite}.

	As in \cite[Theorem 9.3]{fernos_lecureux_matheus18}, we begin by showing that the displacement $d(Z_n(\omega) x, x)$ is well approximated by the horofunctions $h_\xi (Z_n(\omega)x)$. For the remaining of the section, we keep the notations introduced by Theorem \ref{drift}
	
	\begin{prop}
		Let $x \in X$ be a basepoint. Then for $\nu$-almost every $\xi  \in \partial X$, and $\mathbb{P}$-almost every $\omega \in \Omega$, there exists $C > 0 $ such that for all $n \geq 0$ we have 
		\begin{equation}
			|h_\xi (Z_n(\omega)x) - d(Z_n(\omega) x, x)| < C. 
		\end{equation}
	\end{prop}
	
	\begin{proof}
		Because of Proposition \ref{zero measure}, for $\nu$-almost every $\xi \in \bd X$, $d_T(\xi, z^{+}(\omega)) > \pi$. With Theorem \ref{tits}, this implies that for $\nu$-almost every $\xi \in \partial X$, there is a rank one geodesic $\sigma_\xi$ in $X$ joining $\xi $ to $z^{+}(\omega)$. Let $\xi \in \partial X$ such that $d_T(\xi, z^{+}(\omega)) > \pi$, and fix $R>0 $ such that $\sigma_\xi $ does not bound a flat strip of width $R$.  By Lemma \ref{ballm}, there exist neighbourhoods $U$, $V$ of  $\xi $ and $ z^{+}(\omega)$ respectively in $\overline{X}$ such that for all $\xi' \in U$ and $\eta \in V$, there is a geodesic from $\xi' $ to $\eta$, and for any such geodesic $\sigma'$, we have $d(\sigma_\xi(0), \sigma') < R$.  Assume first that $x = \sigma_\xi(0)$. Since $Z_n(\omega)x \rightarrow z^{+}(\omega)$ almost surely, there exists $n_0 $ such that for all $n \geq n_0$, $Z_n(\omega)x \in V$. We are going to show that for all $n \geq n_0 $, $|h^x_\xi (Z_n(\omega)x) - d(Z_n(\omega) x, x)| \leq 2R$. 
		
		Take $(y_p)_p$ a sequence in $X$ converging to $\xi $. There exists $p_0$ such that for all $p \geq p_0$, $y_p \in U$. Fix $n \geq n_0$ and $p \geq p_0$, and define $x' = x'(n,p)$ as the projection of $x$ on the geodesic segment joining $y_{p}$ to $Z_{n}(\omega)x$. By Lemma \ref{ballm}, $d(x', x) \leq R$, hence
		\begin{eqnarray}
			d(y_p, Z_n(\omega)x) & = & d(y_p, x') + d(x', Z_n (\omega)x) \nonumber \\
			& \geq & d(y_p, x) - R + d(x, Z_n (\omega)x) - R \nonumber.
		\end{eqnarray}
		Then for all $p \geq p_0$, $n \geq n_0$, $d(y_p, Z_n(\omega)x) -d(y_p, x) \geq d(x, Z_n (\omega)x) -2R$. If we make $p \rightarrow \infty$, we get that for all $n \geq n_0$,
		\begin{equation*}
			h^x_\xi (Z_n(\omega)x) + 2R \geq d(x, Z_n(\omega)x).
		\end{equation*}
	
		Conversely, for all $n \in \mathbb{N}$, $d(x, Z_n(\omega)x) \geq d(y_p, Z_n(\omega)x) - d(x, y_p)$ hence $d(x, Z_n(\omega)x) \geq h^x_\xi (Z_n(\omega)x) $ by taking the limit. Then there exists $C >0 $ such that for all $n \in \mathbb{N}$, 
		\begin{equation}
			|h^x_\xi (Z_n(\omega)x) - d(Z_n(\omega) x, x)| \leq C \label{bdd}
		\end{equation}
		
		Now if we take a different basepoint $z \in X$,
		\begin{eqnarray}
			d(Z_n(\omega)z, z ) &\leq &d(Z_n(\omega)z, Z_n(\omega)x) + d(Z_n(\omega)x, x ) + d(x, z ) \nonumber \\
			& \leq & d(Z_n(\omega)x, x ) + 2 d(x, z ),  \nonumber
		\end{eqnarray}
		hence $|d(Z_n(\omega)z, z ) - d(Z_n(\omega)x, x )| \leq 2 d(x, z )$.
		Similarly, $|h^x(Z_n(\omega)z) - h^x_\xi(Z_n(\omega)x)| \leq 2 d(x, z )$ and if we change the basepoint, $|h^x_\xi - h^z_\xi | \leq d(z, x)$, so equation \eqref{bdd} does not depend on the choice of the basepoint. 
	\end{proof}
	
	As a consequence, we have the following corollary: 
	
	\begin{cor}\label{sublin approx}
		For every $x \in X$, $\mathbb{P}$-almost surely every $\omega \in \Omega $ and $\nu$-almost every $\xi \in \partial X$, we have that 
		\begin{equation*}
			\lambda = \lim_{n \rightarrow \infty} \frac{1}{n} h_\xi (Z_n(\omega) x).
		\end{equation*}
	\end{cor}

	The rest of the proof is now very similar to what is done in \cite[Theorem 9.3]{fernos_lecureux_matheus18}, which was itself inspired by \cite{benoist_quint}. We include it for completeness. The idea is to apply results about additive cocycles. 
	
	Define  $\check{\mu} $ by $\check{\mu}(g)=  \mu(g^{-1})$. It is still an admissible measure for $G$, so we can apply Theorem \ref{unicite} to find $\nui$ the unique $\mui$-stationary measure on $\overline{X}$. Define $T : (\Omega \times \overline{X}, \mathbb{P} \times \nui ) \rightarrow (\Omega \times \overline{X}, \mathbb{P} \times \nui )$ be defined by $T(\omega, \xi ) \mapsto (S\omega, \omega_0^{-1} \xi)$. 
	The following Proposition is proved in \cite{benoist_quint}. Its key ingredient is the fact that $\nui$ is the unique $\mui$-stationary measure.

	\begin{prop}\label{ergodic}
		The transformation $T$ preserves the measure $\mathbb{P} \times \nui$ and acts ergodically. 
	\end{prop}
	
	\begin{proof}
		Let $\beta = \mathbb{P} \times \nui$. For any bounded Borel function $\psi $ on $\Omega \times \overline{X}$, 
		\begin{equation*}
			\beta(\psi)= \int_\Omega \int_{\overline{X}} \psi(\omega,x) d\mathbb{P}d\nui(x).
		\end{equation*}
		The following computation shows that $T$ is probability  measure preserving. 
		\begin{eqnarray}
			\beta(\psi \circ T)& = & \int_\Omega \int_{\overline{X}} \psi(S \omega,\omega_0^{-1}x) d\mathbb{P}(\omega)d\nui(x) \nonumber \\
			& = &  \int_\Omega \int_{\overline{X}} \psi(S\omega,x) d\mathbb{P}( \omega)d\nui(x) \text{ because $\nui$ is $\mui$-stationary}\nonumber \\
			& = &  \beta(\psi). \nonumber
		\end{eqnarray}
		Let us show that $\beta $ is ergodic. Let $\psi$ be a Borel function  on $\Omega \times \overline{X}$ which is $T$-invariant. 
		
		Let us define by $\pmui$ the Markov operator associated to $\mui $: for all Borel bounded function $f $ on $\overline{X}$, $\pmui f (x) = \int f(gx)d\mui(g)$. It follows that 
		\begin{eqnarray}
			\nui(\pmui f) & = & \int \int f(gx)d\mui(g) d\nui(x) \nonumber \\
			& = & \mui \ast \nui (f) \nonumber \\
			& = & \nui(f) \text{ by $\mui$-stationarity.}
		\end{eqnarray}
		Reversing this computation, it is then equivalent to say that a measure $\nui'$ is a $\mui$-stationary measure on $\overline{X}$ and to say that it is $\pmui$-invariant. Since $\nui$ is the unique $\mui$-stationary measure, it is the only $\pmui$-invariant measure on $\overline{X}$, hence $\nui$ is $\pmui$-ergodic. 
		
		Let $\psi $ be a $T$-invariant bounded Borel function on $\Omega \times \overline{X}$. Denote $\phi (x) = \int \psi(\omega, x) d\mathbb{P}(\omega) d\nui(x)$, which is a bounded Borel function on $\overline{X}$. 
		\begin{eqnarray}
			\pmui \phi (x) &=& \int \int \psi (\omega, gx) d\mathbb{P}(\omega) d\mui(g) \nonumber \\
			&=& \int \int \psi (\omega, g^{-1}x) d\mathbb{P}(\omega) d\mu(g) \nonumber \\
			&=& \int (\psi \circ T) (\omega, x) d\mathbb{P}(\omega) = \phi (x) \nonumber. 
		\end{eqnarray}
	
		Thus $\phi$ is $\pmui$-invariant, hence constant by ergodicity, say equal to $C$. Let $\mathcal{X}_n $ be the $\sigma$-algebra generated by $\mu^{\otimes n} \times \nui$, and $\phi_n = \mathbb{E}[\phi \, | \, \mathcal{X}_n]$, so that the sequence $(\phi_n)$ is a bounded martingale, and then converges to $\psi$ by the martingale convergence theorem. We have by definition 
		\begin{eqnarray}
			\phi_n(\omega_0, \dots, \omega_{n-1}, x) & = & \int \psi((\omega_0, \dots, \omega_{n-1}), \omega, x) d \mathbb{P}(\omega) \nonumber \\
			& = & \int \psi \circ T^n ((\omega_0, \dots, \omega_{n-1}), \omega, x) d \mathbb{P}(\omega) \nonumber \text{ by $T$-invariance}\\
			& = & \int \psi \circ T^n (\omega, \omega_{n-1}^{-1} \dots \omega_0^{-1}x) d \mathbb{P}(\omega) \nonumber\\
			& = & \phi (\omega_{n-1}^{-1} \dots \omega_0^{-1}x) \nonumber \\
			& = & C. \nonumber
		\end{eqnarray}
		Then $\psi $ is also constant. We have proven that $T$ acts ergodically on $\beta$.
		
	\end{proof}

	We can now conclude the proof of Theorem \ref{drift}. 
	
	\begin{proof}[Proof of Theorem \ref{drift}]
		
		Let $x  \in X$ be a base point. Define the function $H : \Omega \times \overline{X} \rightarrow \mathbb{R}$ by 
		\begin{equation*}
			H(\omega, \xi) = h_\xi (\omega_0 x). 
		\end{equation*}
		Recall that $h_\xi $ is 1-Lipschitz on $X$, hence $|H(\omega, \xi)| \leq d(x, \omega_0 x ) $. Since $\mu $ has finite first moment, $\int |H(\omega, \xi)| d\mathbb{P}(\omega)d\nui(\xi) < \infty$. 
		
		Now observe that for all $g_1, g_2 \in G, \, y \in Y$, horofunctions satisfy a cocycle relation: 
		\begin{eqnarray}
			h_\xi (g_1g_2 y ) & = &  \lim_{x_n \rightarrow \xi}d(g_1g_2x , x_n) - d(x_n, x)  \nonumber \\
			& = & \lim_{x_n \rightarrow \xi}d(g_2x , g_1^{-1}x_n) - d(g_1 x, x_n ) +   d(g_1 x, x_n ) - d(x_n, x) \nonumber \\ 
			& = & \lim_{x_n \rightarrow \xi}d(g_2x , g_1^{-1}x_n) - d( x, g_1^{-1}x_n ) +   d(g_1 x, x_n ) - d(x_n, x) \nonumber \\ 
			& = & h_{g_1^{-1}\xi} (g_2 x) + h_\xi (g_1 x). \label{cocycle horof}
		\end{eqnarray}
		Relation \eqref{cocycle horof} gives that 
		\begin{eqnarray}
			h_\xi (Z_n x) =  \sum_{k=1}^{n} h_{Z_k^{-1}\xi} (\omega_k x )  = \sum_{k=1}^{n} H(T^k (\omega, \xi)) \label{transient cocycle}. 
		\end{eqnarray}
		
		By Proposition \ref{sublin approx}, we have that for $\nu$-almost $\xi \in \overline{X}$, and $\frac{1}{n}h_\xi (Z_n x) \rightarrow \lambda$, thus $\frac{1}{n}\sum_{k=1}^{n} H(T^k (\omega), \xi) \rightarrow \lambda$. In the meantime, due to Proposition \ref{ergodic}, we can apply Birkhoff Ergodic Theorem and obtain
		\begin{equation*}
			\frac{1}{n}\sum_{k=1}^{n} H(T^k (\omega, \xi)) \rightarrow \int H(\omega, \xi) d\mathbb{P}(\omega)d\nui (\xi).
		\end{equation*}
		
		Now, by Proposition \ref{sublin approx} together with Theorem \ref{convergence} gives that $h_\xi(Z_n(\omega)x) $ tends to $+\infty$ almost surely. By equation \eqref{transient cocycle}, it means that $\sum_{k=1}^{n} H(T^k (\omega), \xi)$ is a transient cocycle. Now by \cite[Lemma 3.6]{guivarch_raugi85}, we obtain that $\int H(\omega, \xi) d\mathbb{P}(\omega)d\nui (\xi) > 0$. In other words, the drift is positive and we have proven Theorem \ref{drift}. 
	\end{proof}
	
	\subsection{Applications}

	We can now add an application that is a reformulation of \cite[Theorem 2.1]{karlsson_margulis}, now that we know that the drift is positive. It states that we have a geodesic tracking of the random walk. 
	
	\begin{cor}
		Let $G$ be a discrete group and $G \curvearrowright X$ a non-elementary action by isometries on a proper $\cat (0)$ space $X$. Let $\mu \in \prob(G) $ be an admissible probability measure on $G$ with finite first moment, and assume that $G $ contains a rank one element. Let $x \in X$ be a basepoint of the random walk. Then for almost every $\omega \in \Omega$, there is a unique geodesic ray $\gamma^\omega : [0, \infty) \rightarrow X$ starting at $x$ such that 
		\begin{equation}
			\lim_{n\rightarrow \infty} \frac{1}{n} d(\gamma^\omega(\lambda n), Z_n(\omega)x) = 0, 
		\end{equation}
		where $\lambda $ is the (positive) drift of the random walk. 
	\end{cor}
	
	Another application that could be of interest is about boundary theory. The convergence of the random walk stated in Theorem \ref{convergence} provides a natural map 
	\begin{equation}
		z^{+} :  \left\{ \begin{array}{rcr}
			\Omega & \rightarrow & \bd X \nonumber\\
			\omega & \mapsto & z^{+}(\omega). \nonumber
		\end{array} 
		\right.
	\end{equation}
	
	Since for all $n, \omega$, $Z_n(S\omega) = \omega_0^{-1} Z_{n+1}(\omega)$, we have the equivariance property 
	\begin{equation}
		z^{+}(S\omega) = \omega_0^{-1}z^{+}(\omega). \nonumber
	\end{equation}
	
	In other words, $(\bd X, \nu)$ is a $(G, \mu)$-boundary. A natural question is to determine under which conditions $(\bd X, \nu)$ is maximal between $(G, \mu)$-boundaries in the sense of Theorem \ref{poisson boundary}. Now Kaimanovich gave a criterion \cite[Theorem 6.4]{kaimanovitch00}, namely the "strip criterion" for determining whether $(\bd X, \nu)$ is maximal within the category of $(G, \mu)$-boundaries. 
	\newline
	
	A \textit{gauge} on $G$ is an increasing sequence $\mathcal{G} = (\mathcal{G}_k)_k$ exhausting $G$. The \textit{gauge function} associated to $\mathcal{G}$ is the the function 
	\begin{equation*}
		|g|_\mathcal{G} := \min \{k \, : \, g \in \mathcal{G}_k\}.
	\end{equation*} 
	
	\begin{thm}[{\cite[Theorem 6.4]{kaimanovitch00}}]\label{kai strips}
		Let $\mu$ be a probability measure on a countable group $G$ with finite entropy $H(\mu) = - \sum_{g \in G} \mu(g) \log (\mu(g)) < \infty$, and let $(B_{-}, m_-)$ and $(B_{+}, m_+)$ be $(G, \mui)$ and $(G, \mu)$-boundaries respectively. Assume that there exists a gauge $\mathcal{G}$ on $G$ and a measurable $G$-equivariant map $S$ assigning to pairs of points $(b_-, b_+ ) \in (B_-, B_+)$ non-empty "strips" in $G$ such that for all $g \in G$, and $(m_- \otimes m_+ )$-a.e. $(z_-, z_+) \in B_- \times B_+$, 
		\begin{equation}
			\frac{1}{n} \log |S(b_-, b_+) g \cap \mathcal{G}_{|Z_n|}| \underset{n  \rightarrow \infty}{\longrightarrow} 0 \label{strip criterion equation}
		\end{equation}
		in probability with respect to $\mathbb{P}$, then the boundary $(B_+, b_+)$ is maximal.
	\end{thm}
	
	Using this celebrated result, it could be possible to adapt our situation to this context in order to give satisfactory criteria for which $(\bd X, \nu)$ is in fact the Poisson boundary of $(G, \mu)$. If we further assume that the action is proper and cocompact, the criterion is satisfied and it was done by Karlsson and Margulis \cite[Corollary 6.2]{karlsson_margulis}. 
	
	If we do not assume that the action is geometric, we think that Corollary \ref{limit points are rank one} could be useful in order to find the strips required in Theorem \ref{kai strips}, and thus proving the maximality of $(\bd X, \nu)$ as a $(G, \mu)$-boundary. T.~Fern\'{o}s used this kind of strategy in order to give weak conditions under which the Roller boundary of a finite dimensional $\cat(0)$ cube complex is in fact the Furstenberg-Poisson boundary of a random walk on an acting group $G$. Nevertheless, we were not able to determine satisfying assumptions under which Theorem \ref{kai strips} could be applied in our context.

	\bibliographystyle{alpha}
	\bibliography{bibliography}

\end{document}